\documentclass{birkmult}
 \newtheorem{thm}{Theorem}[section]
 \newtheorem{cor}[thm]{Corollary}
 \newtheorem{lem}[thm]{Lemma}
 
 \theoremstyle{definition}
 
 \theoremstyle{remark}
 
 \newtheorem*{ex}{Example}
 \numberwithin{equation}{section}

\newcommand{\Sym}{\ensuremath \mathrm{Sym}}
\newcommand{\ord}{\ensuremath \mathop \mathrm{ord} \nolimits}

\begin{document}
%-------------------------------------------------------------------------
% editorial commands: to be inserted by the editorial office
%
%\firstpage{1}
%\volume{228}
%\DOI{003-0001}
%\seriesextraline{This is the Concrete Title of this Book\br H.E. R and S.T.C. W, Eds.}
%\firstpage{1}
\title[Refinement of Two-Factor Factorizations of LPDOs]
 {Refinement of Two-Factor Factorizations of
  \\a Linear Partial Differential
 Operator\\
 of Arbitrary Order and Dimension
 }
%----------Author 1
\author[Ekaterina Shemyakova]{Ekaterina Shemyakova}

\address{%
Research Institute for Symbolic Computation (RISC) \\
J.Kepler University \\
Altenbergerstr. 69 \\
A-4040 Linz \\
Austria}

\email{Ekaterina.Shemyakova@risc.jku.at}



\date{July 14, 2010}

\begin{abstract} Given a right factor and a left factor of a Linear Partial Differential
Operator (LPDO), under which conditions we can refine these two-factor
factorizations into one three-factor factorization?
This problem is solved for LPDOs of arbitrary order and number of
variables. A more general result for the incomplete factorizations of
LPDOs is proved as well.
\end{abstract}

%%% ----------------------------------------------------------------------
\maketitle
%%% ----------------------------------------------------------------------
%\tableofcontents

\section{Introduction}

The factorization of Linear Partial Differential Operators (LPDOs)
is an essential part of recent algorithms for the exact solution for
Linear Partial Differential Equations (LPDEs). Examples of such
algorithms include numerous generalizations and modifications of the
18th-century Laplace Transformations
Method~\cite{2ndorderparab,ts:genLaplace05,ts:steklov_etc:00,anderson_juras97,anderson_kamran97,athorne1995,Zh-St,St08},
the Loewy decomposition method~\cite{gs,gs:genloewy05,gs:08}, and
others.

The problem of constructing a general factorization algorithm for an
LPDO is still an open problem, although several important
contributions have been made in recent decades, and different
approaches have been applied (see~\cite{gs,LiSchTs:03,LiSchTs02,ts:enumerLODO:96,ts:genLaplace05,obstacle2,
Cluzeau:fact:char:p:2003,Cluzeau:Quadrat:fact:2008,Cluzeau:Quadrat:2006} and many
others).
Many of the recent approaches are concerned, in particular, with
explaining the non-uniqueness of factorization:
(irreducible) factors and the number of factors are
not necessarily the same for two different factorizations of the
same operator. This is commonly illustrated by
the famous example of Landau~\cite{blumberg},
\begin{ex}[Landau]
\begin{eqnarray*}
L &=& \left(D_x + 1 + \frac{1}{x+c(y)} \right) \circ  \left(D_x + 1 -
  \frac{1}{x+c(y)} \right) \circ  \left(D_x + x D_y \right) = \\
  &=&  \left(D_{xx} +x D_{xy} +D_x +(2+x)D_y \right)\circ  \left(D_x + 1 \right) \ ,
\end{eqnarray*}
where the second-order factor in the second factorization is
hyperbolic and is irreducible.
\end{ex}

On the other hand, for some classes of LPDOs factorization is unique. For
example, there is no more than one factorization that
extends a factorization of the principal symbol of the operator into
co-prime factors~\cite{gs}.
Algebraic theories have been introduced to explain this phenomenon theoretically; see
Tsarev~\cite{ts:steklov_etc:00}, Grigoriev and Schwarz~\cite{gs:08} and most recently
Cassidy and Singer~\cite{Cas:Singer2011}.

Some important methods of exact integration, for example, the Loewy decomposition methods mentioned
above, require LPDOs to have a number of
different factorizations of certain types. Also completely reducible
LPDOs introduced in~\cite{gs}, which become significant as the
solution space of a completely reducible LPDO coincides with the sum
of those of its irreducible right factors may require a number of
right factors.

In earlier work~\cite{parameters} we have exhaustively studied
families of factorizations for operators up to order $4$, and described
when the same operator has multiple factorizations of
\emph{the same factorization type},
to be more specific, when there exist an infinite number of factorizations of
the same factorization type, meaning having the same symbols of the factors.
The first non-trivial example of such families
of order $4$ has been found:
\begin{ex}~\cite{parameters}
The following is a fourth-order irreducible family of factorizations:
\[
D_{xxyy}=\Big(D_x + \frac{\alpha}{y + \alpha x + \beta} \Big) \Big(
D_y + \frac{1}{y + \alpha x + \beta} \Big) \Big(D_{xy} - \frac{1}{y
+ \alpha x + \beta} (D_x + \alpha D_y) \Big),
\]
where $\alpha, \beta \neq  0 $. Note that the first
two factors commute. So the operator $D_{xxyy}$ has a family of factorizations,
and every factorization of the family is of the same
factorization type $(X)(Y)(XY)$, that is the highest order terms in
the first, the second and the third factors are $D_x$, $D_y$ and $D_{xy}$ correspondingly.
\end{ex}

In recent work~\cite{multiple_fact} non-uniqueness of a different kind is addressed.
There, we considered factorizations of \emph{different factorization types}, and
by using invariants proved that a third-order bivariate operator $L$
has a first-order left factor of the symbol $S_1$ and a first-order
right factor of the symbol $S_2$, where $\gcd(S_1,S_2)=1$ if and
only if it has a complete factorization of the type $(S_1)(T)(S_2)$,
where $T=\Sym(L)/(S_1 S_2)$. Further investigations in the same
paper show that a third-order bivariate operator $L$ has a
first-order left factor $F_1$ and a first-order right factor $F_2$
with $\gcd(\Sym(F_1),\Sym(F_2))=1$ if and only if $L$ has a
factorization into three factors, the left one of which is exactly
$F_1$ and the right one is exactly $F_2$.

\begin{ex}~\cite{multiple_fact}
The existence of two factorizations for an LPDO,
\[
(D_x + x) \circ (D_{xy}+ yD_x +y^2 D_y + y^3)=A=(D_{xx}+ (x+y^2)D_x
+ xy^2) \circ (D_y + y)
\]
implies the existence of the ``complete'' factorization of $A$,
\[
A=(D_x + x) \circ (D_x +y^2) \circ (D_y +y) \ .
\]
\end{ex}

On the other hand,
if the condition $\gcd(\Sym(F_1),\Sym(F_2))=1$ fails, then it can
happen that the ``complete'' factorization does not exist.

\begin{ex}~\cite{multiple_fact}
\[
(D_x D_y + 1 ) \circ (D_x + 1) = (D_x + 1) \circ (D_x D_y + 1 ) \ ,
\]
while $D_x D_y +1$ has no factorization at all.
\end{ex}

In the present paper we have generalized the result of~\cite{multiple_fact}
to the case of LPDOs of arbitrary order and of arbitrary
dimension. Moreover, a more general statement has been formulated and proved
for incomplete factorizations of LPDOs. We describe the results in terms of common obstacles,
which we have introduced in~\cite{obstacle2}.

\section{Preliminaries}

Consider a field $K$ of characteristic zero with commuting
derivations $\partial_1, \dots , \partial_n$, and the corresponding
non-commutative ring of linear partial differential operators
(LPDOs) $K[D]=K[D_1, \dots, D_n]$, where $D_i$ corresponds to the
derivation $\partial_i$ for all $i \in \{1, \dots , n\}$. In $K[D]$
the variables $D_1, \dots, D_n$ commute with each other, but not
with elements of $K$. We write multiplication in $K[D]$ as
``$\circ$''; i.e. $L_1\circ L_2$ for $L_1,L_2 \in K[D]$. Any
operator $L \neq 0 \in K[D]$ has the form
\begin{equation} \label{op:general_form_inKD}
L = \sum_{|J| =0}^d a_{J} D^J \ , a_J \in K \ ,
\end{equation}
where $J=(j_1, \dots , j_n)$ is a multi-index in $\mathbb{N}^n$,
$|J|=j_1+ \dots + j_n$, and where $D^J=D_1^{j_1} \dots D_n^{j_n}$.
Further, there exists some $J$, with $|J|=d$, such that
$a_J \neq 0$. Then $d$ is the order of $L$. For the case $L=0$ we
define the order as $- \infty$.

When considering the bivariate case  $n=2$, we use the following
formal notations: $\partial_1=\partial_x$, $\partial_2=\partial_y$,
$\partial_1(f)=f_x$, $\partial_2(f)=f_y$, where $f\in K$, and
correspondingly $D_1\equiv D_x$, and $D_2 \equiv D_y$ for ease of
notation.

For an operator $L \neq 0$ of the form~(\ref{op:general_form_inKD}) the
homogeneous commutative polynomial
\begin{equation} \label{sym:general_form_inKD}
\Sym(L) =  \sum_{|J| = d} a_J X^J
\end{equation}
in formal variables $X_1, \dots , X_n$ is called the
(\emph{principal}) \emph{symbol}, and if $L=0$, the symbol is
defined to be zero. Vice versa, given a homogeneous
commutative polynomial $S \in K[X]$ in the
form~(\ref{sym:general_form_inKD}), we define the operator
$\widehat{S }\in K[D]$ as the result of substituting $D_i$ for each
variable $X_i$.

\section{Main Result}
\label{sec:main}

Since for two LPDOs $L_1, L_2 \in K[D]$ we have $\Sym(L_1 \circ L_2)
= \Sym(L_1) \cdot \Sym(L_2)$, any factorization of an LPDO extends
some factorization of its symbol. In general, if $L \in K[D]$ and
$\Sym(L)=S_1 \dots S_k$, let us say that the
factorization
\[ L=F_1 \circ \dots  \circ F_k\ , \quad \Sym(F_i)=S_i,\  \forall i
\in \{1, \dots, k\}\ ,\]
is of the factorization type $(S_1)\dots(S_k)$.

For the second-order hyperbolic LPDOs, which have normalized form
\begin{equation} \label{op:LL}
L = D_x D_y + a D_x + b D_y + c \ ,
\end{equation}
where $a,b,c \in K$, it is common to consider their incomplete
factorizations:
\[ L= (D_x + b) \circ (D_y + a) + h = (D_y + a) \circ (D_x + b) + k\ ,\]
where $h= c-a_x - ab$ and $k= c-b_y - ab$
are invariants of (\ref{op:LL}) with respect to
gauge transformations, $L \rightarrow g^{-1} L g$, $g \neq 0, g \in
K$ and are called the Laplace invariants. This is an element in the
foundation of the classical Laplace-Darboux-Transformations
Method~\cite{Darboux2}.

In~\cite{obstacle1,obstacle2} a generalization of this idea is
suggested. Thus, for $A \in K[D]$ with $\Sym(A)=S_1 \dots S_k$, we
call~\cite{obstacle1,obstacle2} an LPDO $R \in K[D]$ a \emph{common
obstacle} to factorization of the type $(S_1)(S_2) \dots (S_k) $ if
there exists a factorization of this type for the operator $A-R$, and
$R$ has minimal possible order.

The following example demonstrates different possibilities
for common obstacles and incomplete factorizations.

\begin{ex} Consider the LPDO
\begin{equation} \label{op:A4}
  A_4 = D_x^2D_y^2 + D_x + D_y + 1 \ .
\end{equation}

\emph{1. Unique common obstacle and unique incomplete factorization.
} Consider factorizations of $A_4$ of the factorization type $(X^2)(Y^2)$.
Assume that the order of common obstacles is one or less (if we come
to a contradiction, we have to search then for higher-order common
obstacles), and search for common obstacles in the form $R_1=p_1 D_x
+ q_1 D_y + r_1$, where $p_1, q_1, r_1 \in K$. Thus, for some
$l_{10}, l_{01}, l_{00}, f_{10}, f_{01}, f_{00} \in K$ we have
\[ A_4 = (D_x^2 + l_{10}D_x + l_{01} D_y + l_{00}) \circ (D_y^2 +f_{10}
D_x + f_{01} D_y + f_{00}) + R_1\ .\]
Comparing the corresponding coefficients we have
$l_{10} = l_{01} = l_{00} = f_{10} = f_{01} = f_{00} = 0$,
$p_1=q_1=r_1=1$, that is, there is a unique common obstacle and a unique incomplete factorization
of the factorization type $(X^2)(Y^2)$,
\[
A_4 =D_x^2 \circ D_y^2 + D_x+D_y+1 \ .
\]

\emph{2. Infinitely many common obstacles and incomplete factorizations.
} Consider factorizations of $A_4$ of the factorization type $(X)(XY^2)$. Again
assume that the order of common obstacles is one or less, and search
for common obstacles in the form $R_2=p_2 D_x + q_2 D_y + r_2$,
where $p_2, q_2, r_2 \in K$. Thus, for some $m_{00}, g_{ij} \in K$
we have $A_4 = (D_x+ m_{00}) \circ (D_xD_y^2 + \sum_{i+j=0}^2 g_{ij} D_x^i
D_y^j) + R_2$. Comparing the corresponding coefficients we have
$g_{02}=-m_{00}$, $g_{20} = g_{11} = g_{10} = g_{01} = 0$,
$q_2=1$, $p_2 = 1 -g_{00}$, $r_2 = 1 - m_{00} g_{00} - g_{00x}$,
while $m_{00}$ satisfies $m_{00}^2 + m_{00x} = 0$,
and $g_{00}$ is a free parameter. Thus, we have
\[
A_4 = (D_x+ m_{00}) \circ (D_xD_y^2 -m_{00} D_y^2 + g_{00}) + (1 -g_{00}) D_x + D_y
+ 1 - m_{00} g_{00} - g_{00x} \ ,
\]
and the order of common obstacles is $1$.

\emph{3. Unique common obstacle and infinitely many incomplete factorizations.
}
Consider factorizations of $A_4$ of the factorization type $(X)(X)(Y^2)$.
We search for common obstacles in the form $R_3=p_3 D_x + q_3 D_y +
r_3$, where $p_3, q_3, r_3 \in K$. Thus, for some $m_3, n_3, a_3,
b_3, c_3 \in K$ we have
\[ A_4 = (D_x + m_3) \circ (D_x + n_3) \circ (D_y^2 + a_3 D_x + b_3 D_y
+ c_3) + R_3\ .\]
Equating the corresponding coefficients we have $n_3 = - m_3$, $a_3 = b_3 = c_3 = 0$,
$p_3=q_3=r_3 = 1$, and $m_3$ satisfies
$m_3^2 + m_{3x} = 0$,
that is we have a unique common obstacle, but
incomplete factorizations can be different:
\[
A_4 = (D_x + m_3) \circ (D_x - m_3) \circ D_y^2 + D_x + D_y + 1 \ .
\]
\end{ex}

The following lemma is used for the proof of Theorem~\ref{thm:main_obstacle}.

\begin{lem}[Division lemma] \label{lem:division}
Let $L, M \in K[D]$ and $\Sym (L)$ is divisible by $\Sym (M)$,
then there exist $N, R \in K[D]$ such that
\[
L = M \circ N + R \ ,
\]
where either $R=0$,  or $\Sym (R)$ is not divisible by $\Sym (M)$.
Here $R$ is the remainder of the incomplete factorization.
\end{lem}
\begin{proof} Let $\Sym (L) = S_1 S_2$, $\Sym (M) = S_1$.
Construct a finite sequence of $n$ (for some $n$) incomplete factorizations of $L$ of
the form $L = M \circ N_i +Q_i$, where $\Sym (N_i) = S_2, 1 \le i \le n$, and
$\Sym (Q_n)$ is either zero or
not divisible by $\Sym(M)$. Start with $N_1 = \widehat{S_2}$ and
let $Q_1 = L-M \circ N_1$. If $\Sym(Q_1)$ is either zero or not divisible by $\Sym(M)$,
we stop and let $N = N_1$ and $R = Q_1$. Otherwise,
let $T_1 = \Sym(Q_1)/ \Sym(M)$ and let $N_2 = N_1 + \widehat{T_1}$, and let
$Q_2 = Q_1 -M \circ \widehat{T_1}$ (which implies $Q_2 = L -M \circ N_2$).
If $\Sym (Q_2)$ is either zero or not divisible by $\Sym (M)$, we stop and let $N = N_2$, $R = Q_2$.
Otherwise, we continue in the same manner. Since we clearly have $\ord (Q_2) <
\ord (Q_1) < \ord (L)$, and in general, $\ord (Q_{i+1}) < \ord (Q_i)$, this process must
stop after a finite, say $n$, number of steps, and we have
$L = M \circ N + R$,
where $\Sym(N) = S_2$ and $\Sym (R)$ is either zero or is not divisible by $\Sym (M)$.
\end{proof}

Let $A \in K[D]$ and $\Sym(A)=S_1 \cdot S_2 \cdot S_3$.
It is easy to see that every common obstacle to factorization of the type
$(S_1)(S_2)(S_3)$ is the remainder for some incomplete factorization of the type
$(S_1 S_2)(S_3)$ and so it is for some incomplete factorization of the type $(S_1)(S_2S_3)$
(the order of the common obstacles for a factorization of the type
$(S_1 S_2)(S_3)$ (resp. $(S_1)(S_2S_3)$) can be smaller that of
the common obstacles for a factorization of the type $(S_1)(S_2)(S_3)$.
In general, the inverse statement is not true for it is more
difficult to find a factorization into more factors.

The following theorem states that under some conditions, common
obstacles to factorization into two factors are the same as those
into three factors.

\begin{thm}
\label{thm:main_obstacle}
Let $A \in K[D]$ and suppose $\Sym(A) = S_1 S_2 S_3$, with $\gcd(S_1,S_3)=1$.
Let $U$, $V$ and $W$ be respectively the sets of common obstacles to factorizations of $A$
of the types $(S_1)(S_2S_3)$, $(S_1 S_2)(S_3)$, and $(S_1)(S_2)(S_3)$.
Suppose $V$
(resp. $U$) is non-empty and the order of common obstacles in $V$ is less
than $\ord (S_3)$. Then $W$ is non-empty and $V = W$.
\end{thm}
\begin{proof}
Let $R_1 \in V$ be any common obstacle of type $(S_1 S_2)(S_3)$.
Then we have $\ord (R_1) < \ord (S_3)$. Let $L, F \in K[D]$ be such
that
\begin{equation} \label{eq_obst_thm_1}
A = L \circ F + R_1 \ ,
\end{equation}
where $\Sym(L) = S_1 S_2$, $\Sym(F)=S_3$. Similarly, let $R_2 \in U$,
$\ord (R_2) < \ord (S_3)$ be a fixed remainder with respect to an
incomplete factorizations of $A$
of type $(S_1)(S_2S_3)$. Let $M, G \in K[D]$ be such that
\begin{equation}
\label{eq_obst_thm_2}
A = M \circ G + R_2  \ ,
\end{equation}
where $\Sym(M) = S_1$, $\Sym(G)=S_2 S_3$.

By Division Lemma~\ref{lem:division}, there exist $N, R \in K[D]$ such
that
\begin{equation} \label{eq_obst_thm_3}
L = M \circ N + R \ ,
\end{equation}
where $\Sym(N) = S_2$ and $\Sym (R)$ is either zero or is not divisible by $\Sym (M)$.

We now claim that $R = 0$. Combining~(\ref{eq_obst_thm_1}), (\ref{eq_obst_thm_2}), (\ref{eq_obst_thm_3}),
we have
\[
(M \circ N + R) \circ F + R_1 = M \circ G + R_2 \ ,
\]
\[
R \circ F + R_1 -R_2 = M \circ (G - N \circ F) \ .
\]
Since the orders of $R_1, R_2$ are both less than the order of $F$, if $R$ were not
zero, the symbol on the left side of the last equation would be that of $R \circ F$, which
would imply that $\Sym(R)$ is divisible by $\Sym (M)$ because $\gcd(S_1, S_3) = 1$.
Hence $R = 0$, showing that $A = M \circ N \circ F + R_1$ is an incomplete factorization
of type $(S_1)(S_2)(S_3)$ with remainder $R_1$. We now show $R_1$ is a common obstacle
for that type. This follows easily since if $A = M_0 \circ N_0 \circ F_0 +R_0$
is any incomplete factorization of that type, then $A = (M_0 \circ N_0) \circ F_0+R_0$
is one of type $(S_1S_2)(S_3)$ and hence $\ord (R_0) \ge \ord (R_1)$.
This completes the proof that $V \subseteq W$, which is
thus non-empty also. If furthermore, $R_0 \in W$, then since we have shown that
$R_1 \in W$ for any $R_1 \in V$, it follows that $\ord (R_0) \le \ord (R_1)$
and hence $R_0$ is also of minimal order as a remainder of type $(S_1S_2)(S_3)$,
or in other words, $R_0 \in V$. This shows that $V = W$.
\end{proof}

\begin{ex} In Theorem~\ref{thm:main_obstacle}, take $A$ to be
$A_4$ from~(\ref{op:A4}), and take $S_1=X$, $S_2=X$, and $S_3=Y^2$.
As we showed in the examples before Theorem~\ref{thm:main_obstacle},
the orders of common obstacles of the types $(S_1 S_2)(S_3)$ and
$(S_1)(S_2 S_3)$ are $1$, which is less then the order of $S_3$. The
theorem implied that the sets of common obstacles to factorization
of the types $(S_1 S_2)(S_3)$ and $(S_1)(S_2)(S_3)$ are the same,
which accords with our computations in the examples
before Theorem~\ref{thm:main_obstacle}.
\end{ex}

\begin{ex}[Assumptions on the orders of common obstacles are necessary]
Consider operator $A_4$ from (\ref{op:A4}), where $\Sym (A) = X^2 Y^2$.
Let $S_1 = X^2$, $S_2 = S_3 = Y$. Then $\gcd(S_1, S_3) = 1$.
It was shown (example before Theorem~\ref{thm:main_obstacle})
that with respect to the type $(S_1)(S_2S_3) = (X^2)(Y^2)$, the
operator $R_2 = D_x + D_y + 1$ is the unique common obstacle. Using similar methods, it can
be shown that with respect to the type $(S_1S_2)(S_3) = (X^2Y)(Y)$, the operator
$R_1 = D_x +1$ is a common obstacle. Here the hypothesis of
Theorem~\ref{thm:main_obstacle} is not satisfied, because $\ord(R_1) = \ord(R_2) = \ord(S_3)
= 1$. It can also be shown that with respect to the type $(S_1)(S_2)(S_3) = (X^2)(Y)(Y)$, the
only common obstacle is $R_0 = D_x + D_y + 1$. Clearly, $R_1 \neq R_0$ cannot
be a common obstacle of type $(S_1)(S_2)(S_3)$. We note that $N = D_y$, $R = 1$ in
this example.
\end{ex}

\begin{cor} \label{thm:main} Let, in $K[D]$, an LPDO $A$ have two factorizations
into two factors:
\[
L \circ F = A = M \circ G \quad (\text{or} \   F \circ L = A = G
\circ M ) \ ,
\]
where $\gcd(\Sym(F),\Sym(M))=1$. Then there is a factorization of
$A$ into three factors:
\[
A=M \circ N \circ F \quad (\text{or} \   A = F \circ N \circ M )
\]
 for some $N \in K[D]$.
\end{cor}
\begin{proof} The first statement is implied from that of
Theorem~\ref{thm:main_obstacle}. For the second (the one which is in
the brackets) we apply properties of the formal adjoints of LPDOs.
\end{proof}

\begin{ex}[Fourth Order LPDO] Let
\begin{eqnarray*}
 L &=& D_x^3 + (1+x) D_x^2 D_y + x D_x D_y^2 - x^2 D_x^2 - x^3D_x D_y  \\
&& \quad +(1-4 x) D_x + (x-2 x^2) D_y - 2\ ,
\end{eqnarray*}
and
$F= D_y + x^2$,
$M = D_x + xD_y$, and
\[ G= D_x^2 D_y + D_x D_y^2 + x^2 D_{xx} + (4x-x^4) D_x + D_y -4
x^3+x^2+2\ .\]
Then $L \circ F = M \circ G$, meaning that we have two different
factorizations into two factors for the LPDO $A=L \circ F$.
Moreover, one can find an
LPDO $N$ such that $L=M \circ N$. Explicitly, $N=D_x^2 + D_xD_y - x^2 D_x -2x+1$.
Then $A=M \circ N \circ F$, meaning that $A$ has a factorization into three
factors.
\end{ex}

\begin{ex}[Multidimensional LPDO] We have $L \circ F = M \circ
G$ for $L= D_x D_y + s D_x + t s + s_x$, $F= D_z + b$,
$M = D_x + t$, $G= D_x D_z + b D_y + s D_z + s b + b_y$.
It is also easy to see that $L= M \circ N$, where $N=D_y + s$.
\end{ex}

\begin{ex}[Condition $\gcd(\Sym(F),\Sym(M))=1$ is necessary for
Theorem~\ref{thm:main}]
Consider
$L = D_x D_y + \frac{1}{1-x} D_x + x D_y + \frac{2-x}{(x-1)^2}$,
$F = D_x + \frac{x}{x-1}$, $M = D_x + 1$,
$G = D_x D_y + \frac{1}{1-x} D_x + \frac{x^2-x+1}{x-1} D_y -
\frac{x}{(x-1)^2}$,
for which $L \circ F = M \circ G$. Here $\Sym(L)$ is divisible by
$\Sym(M)$, but condition $\gcd(\Sym(F),\Sym(M))=1$ fails. On the
other hand, the Laplace invariants for
LPDO $L$ are $h=-1$, $k=-\frac{-2x+2+x^2}{(x-1)^2} \neq 0$,
and, therefore, $L$ has no factorization.
\end{ex}

\section{Conclusions}
The main result of the paper formulated
in Theorem~\ref{thm:main_obstacle} provides a simplification of the overall picture of
factorization of LPDOs.

\section{Appendix}

Below is an example of how the direct approach and
the approach based on~Theorem~\ref{thm:main_obstacle}
are different when it comes to computations.

Let us search for factorizations of the type $(X)(XY)(Y)$
for a bivariate fourth-order LPDO, $A=D_x^2D_y^2 + \sum_{i+j=0}^3 a_{ij} D_x^i D_y^j$, $a_{ij} \in
K$. The direct approach considers $A=M \circ N \circ F$ for some $M=D_x + m$, $N=D_x D_y + n_{10} D_x
+ n_{01} D_y + n_{00}$, $F= D_y + f$, where $m, n_{10}, n_{01},
n_{00}, f \in K$. Equating the corresponding coefficients, we have $a_{30}=a_{03}=0$,
$n_{10} = a_{21}-f$,
$n_{01} = a_{12}-m$, $n_{00} = m f-ma_{21}-f a_{12}-f_x-(a_{21})_x+a_{11}$,
and
\begin{equation} \label{eq:last_sec:eq_for_f_and_m_separately}
\left.
  \begin{array}{lll}
0&=& 2 f_{xy}-f^2 a_{12}-4 f_x f+f a_{11}+2 f_x a_{21}+f_y a_{12}-a_{10} \ , \\
0&=& f_y-f^2+f a_{21}-a_{20} \ , \\
0&=& m a_{11}-m^2 a_{21}-2 m a_{21x}-m_x a_{21}-a_{21xx}+a_{11x}-a_{01} \ , \\
0&=& m a_{12}-m^2-m_x+a_{12x}-a_{02} \ .
  \end{array}
\right\}
\end{equation}

\begin{eqnarray}
0&=& f_{xxy}-a_{00}+m f_y a_{12}-m^2 f_y-2 f_x^2-2 f_{xx} f+f_{xx}
a_{21}+ \label{eq:last_sec:eq_for_fm} \\
 && +f_y a_{12x}+ f_{xy} a_{12} -m^2 f a_{21}+m^2 f^2+f^2 m_x-f^2 a_{12x}- \nonumber \\
 && -f a_{21xx} +f a_{11x}+f_x a_{11}- m f^2 a_{12} -2 m f a_{21x}+m f a_{11}- \nonumber \\
&& -f m_x a_{21} -2 f f_x a_{12} -f_y m_x
 \ ,  \nonumber
\end{eqnarray}
An approach based on Theorem~\ref{thm:main} considers
$A=L \circ F = M \circ G$
for some $L=D_x^2 D_y + \sum_{i+j=0}^2 l_{ij} D_x^i D_y^j$, $F= D_y
+ f$, $M=D_x + m$, $G=D_x D_y^2 + \sum_{i+j=0}^2 g_{ij} D_x^i
D_y^j$, where $l_{ij}, f, m, g_{ij} \in K$.
$A=L \circ F$ implies $a_{30}=0$ and
$l_{20} = a_{21}-f$, $l_{02} = a_{03}$, $l_{11} = a_{12}$,
$l_{10} =a_{11} - a_{12} f - 2 f_x$, $l_{01} = a_{02}-a_{03}f$,
$l_{00} = a_{03} f^2- f a_{02}-a_{12} f_x-2 a-{03} f_y- f_{xx}+a_{01}$,
while $A=M \circ G$ implies $a_{03}=0$, and
$g_{20} = 0$, $g_{11} = a_{21}$, $g_{02} = a_{12}-m$, $g_{10} =
a_{20}$, $g_{01} = a_{11} - m a_{21} - a_{21x}$, $g_{00}
= a_{10}-m a_{20} - a_{20x}$. The remaining conditions are
\[
\text{conditions} \ (\ref{eq:last_sec:eq_for_f_and_m_separately}) \
,
\]
and two new conditions:
\begin{eqnarray}
0&=& a_{00}-m a_{10}+m^2 a_{20}+2 m a_{20x}-a_{10x}+m_x a_{20}+a_{20xx} \ , \label{eq:last_sec:eq_m}\\
0&=& f_{xxy}-f^2 a_{02}-2 f_x a_{12} f-2 f_{xx} f+f a_{01}-2
f_x^2+f_x a_{11}+ \label{eq:last_sec:eq_f} \\
&& + f_y a_{02} +f_{xx} a_{21}+a_{12} f_{xy}-a_{00} \ . \nonumber
\end{eqnarray}

Thus, when algebraic manipulations only are used the difference
between the two approaches applied to the given problem is as follows.
Instead of the non-linear
Partial Differential Equation (PDE) in two unknown variables $f$ and
$m$, (\ref{eq:last_sec:eq_for_fm}) that we have in the first
(direct) approach, the second approach implies a non-linear PDE in
variable $f$, (\ref{eq:last_sec:eq_f}) and another one in variable
$m$, (\ref{eq:last_sec:eq_m}). In other words, the second approach
gives separation of variables.

\subsubsection*{Acknowledgments.}
The author was supported by the Austrian Science Fund (FWF) under
project DIFFOP, Nr. P20336-N18.

% ------------------------------------------------------------------------
\end{document}